\documentclass[runningheads]{llncs}
\usepackage[T1]{fontenc}
\usepackage{graphicx}
\usepackage{MnSymbol}
\usepackage{eufrak}
\usepackage{float}
\usepackage{hyperref}
\usepackage{xcolor}
\usepackage{array}
\usepackage{pgf}
\usepackage{tikz}
\usepackage{lscape}
\usepackage[section]{placeins}
\usetikzlibrary{positioning,arrows,automata}
\usepackage{slashed}
\usepackage{booktabs}
\usepackage{subfigure}
\usepackage{wasysym}
\usepackage{setspace}
\usepackage{enumitem}
\usepackage{tabularx}
\usepackage{linguex}
\usepackage{multirow}
\usepackage{multicol}

\newcommand{\IO}{{\mathrel{\textbf{O}\kern-0.55em\raisebox{0.2ex}{\scalebox{0.6}{\textbf{i}}}}\kern0.4em}}
\newcommand{\IP}{{\mathrel{\textbf{P}\kern-0.5em\raisebox{0.2ex}{\scalebox{0.6}{\textbf{i}}}}\kern0.4em}}

\renewcommand{\S}{\mathcal{S}}
\newcommand{\Sig}{\mathcal{S}}
\newcommand{\NV}{\mathcal{V}}
\newcommand{\XV}{\mathcal{U}}

\newcommand{\M}{\mathcal{M}}
\newcommand{\F}{\mathcal{F}}

\newcommand{\X}{\overrightarrow{X}}
\newcommand{\Y}{\overrightarrow{Y}}
\newcommand{\x}{\overrightarrow{x}}
\newcommand{\y}{\overrightarrow{y}}
\newcommand{\U}{\overrightarrow{U}} 
\renewcommand{\u}{\overrightarrow{u}}
\newcommand{\zi}{\overrightarrow{z}}
\newcommand{\Z}{\overrightarrow{Z}}

\newcommand{\A}{\mathcal{A}}

\newcommand{\R}{\mathcal{R}}

\newcommand{\tuple}[1]{\langle #1 \rangle}

\renewcommand{\int}[1]{#1\boxright}
\newcommand{\syndirparof}{\rightsquigarrow}

\newtheorem{Def}{Definition}

\begin{document}
\title{A logic for instrumental obligation}
%
%
\author{Jialiang Yan\inst{1} \and
Qingyu He\inst{2}}
%
\authorrunning{Yan and He}
%
\institute{Tsinghua University, China \\
\email{jialiang.yann@gmail.com}\\
\and  Tsinghua University, China
\\ 
\email{qingyuhethu@gmail.com}}
\maketitle              
\begin{abstract}
This paper develops a logic based on causal inferences to formally capture the concept of instrumental obligation. We establish a causal deontic model that extends causal models with priority structures, allowing us to represent both the instrumental and deontic aspects of an obligation. In this framework, instrumental obligation is defined as a derived notion through intervention formulas of causal reasoning, where an action is considered obligatory if it is the best way to achieve the goal. We provide a sound and complete axiomatic system and show that the logic is NP-complete. The concept of instrumental permission is also taken  into account in the model.

\end{abstract}
\section{Introduction}\label{intro}

\textit{Instrumental obligation} means that obligatory actions result in beneficial outcomes for the agent. In this context, the right actions are those that help achieve the agent's desired goals or plans. 

In moral philosophy and meta-ethics, instrumental obligation is often discussed in the context of \textit{hypothetical imperatives} (e.g. \cite{Hill1973-EHITHI,foot1972morality,schroeder2005hypothetical}), a key concept in Kantian ethics. Hypothetical imperatives involve taking the necessary means to achieve a certain end, applicable to agents who hold that end. Thus the obligation is understood as an instrumental relation between the action and the agent. In addition, in decision-making theories, instrumental obligation is closely linked to the idea of instrumental responsibility (see \cite{spitzeck2013normative,hof2017instrumental}), which suggests that responsible actions, particularly in business, are ultimately profitable.

Deontic logicians have also explored instrumental obligation by modeling the underlying reasons for obligations. For instance, in \cite{van2014priority}, the authors discuss the reasons behind obligations using priority graph structures in preference logic (see \cite{liu2011reasoning,liu2011two} among others). \cite{Giordani2021} introduces a system of deontic logic that grounds obligations in reasons, providing a framework that integrates both standard and neighborhood semantics to model basic deontic reasoning.

In this paper, we aim to develop a logic that models instrumental obligation and the reasoning about it. Our understanding of instrumental obligation follows the perspective put forth by Stephen Finlay  (\cite{finlay2009oughts,finlay2016oughts}). Finlay argues that, in natural language, instrumental obligation differs from normal conditional obligation. So the semantics for conditional sentences is not enough to account for the instrumental reading of the sentences. This distinction emphasizes the complexity and challenges in providing a reductive analysis of instrumental obligation, suggesting the need for a distinct framework to understand it (for further philosophical discussion, see \cite{schroeder2004scope,schroeder2005hypothetical,setiya2007cognitivism,finlay2009against}). Consider the normal conditional sentence in \Next.

\ex. \label{condi1} $\#$ If Max is trying to evade arrest, he ought to mingle with the crowd. \hfill{(\cite{finlay2009oughts})}

This sentence suggests that, given the condition of evading arrest, Max ought to act in a certain way. However, the instrumental use of obligation just means that if Max performs the action, he will achieve the goal of evading arrest. To clarify this distinction, Finlay proposed that all instrumental ought-propositions should be expressed using an "in order that..." modifier, which makes the purpose of the action explicit.

\ex. \label{condi2} In order to evade arrest, Max ought to mingle with the crowd. \hfill{(\cite{finlay2009oughts})}

Building on Finlay's analysis, we propose a logic that uses causal inferences to formally represent both goals and obligations. We argue that the relationship between a goal and an action can be captured through causal relationships, asserting that an action is obligatory if it is the best way to achieve the goal. We then established a causal deontic model and develop an axiomatic system. We show that the logic is both sound and complete and explore its complexity.

The paper is structured as follows: In Section \ref{SFIO}, we discuss the semantic foundations of instrumental obligation. Section \ref{prelimaries} introduces the necessary preliminaries on causality and ordering. In Section \ref{CDM}, we establish the causal deontic model, and in Section \ref{IOI}, we define the instrumental obligation modality. The logic is developed and its properties are proved in Section \ref{CIO}. Finally, we conclude in Section \ref{conclu}.

\section{Semantic foundation of instrumental obligation}\label{SFIO}

\subsection{Instrumental relation as a causal relation}

Instrumental obligation, typically expressed as ``In order to..., you ought to...", might initially appear to be a form of conditional obligation, as the antecedents reflect agents' goals or plans, while the consequents specify the actions they ought to take. However, as discussed in Section \ref{intro}, instrumental meaning of obligations are not genuinely conditional obligations, and the semantics of conditional sentences are insufficient to account for them.
If we treat instrumental obligation merely as conditional propositions it will raise concerns about the moral meaning of obligation. For instance, consider the sentence \ref{condi1}, repeated here as \ref{condi1'}:

\ex. \label{condi1'} $\#$ If Max is trying to evade arrest, he ought to mingle with the crowd.

In this case, claiming that Max is obligated to mingle with the crowd, given his goal of evading arrest, seems illogical because the action is not moral. The modal ``ought" in such a conditional statement means a specific action under certain conditions, but it fails to capture the true instrumental meaning of obligation—namely, executing the action to achieve a particular goal. This issue, known as the detachment problem, is also discussed in \cite{finlay2009oughts}. 

So, what is the instrumental meaning of obligation? In conditional sentences, the relationship between the goal and the obligatory action implies that the action is conditional on the goal. However, the instrumental meaning suggests that the action leads to the goal. It asserts that what you ought to do is aimed at achieving the goal. 

To represent this relationship, we propose interpreting the instrumental relation as a causal relation. The goal in the antecedent is understood as the outcome of executing the action. 

When we use conditional forms to describe actions, we are explaining the reasons for the action, essentially the conditions provide the reasons for why the action should be performed. This philosophical view can be traced back to Davidson (\cite{Davidson1963}). However, instrumental obligation statements serve a different purpose. They indicate the destination of doing the action rather than the reasons behind it. The action is executed \textit{for} a specific destination, not \textit{because of} certain conditions. In this sense, causality plays a role in linking the action to its goal. Therefore the instrumental meaning can be expressed by causality.

\subsection{Ordering source of deontic modals}

In standard deontic logic, it is common to model obligation, permission, and prohibition using an \textit{ideal ordering} system (see \cite{Hansson1969}), which ranks possible worlds or actions based on how well they satisfy moral requirements. Similarly, in formal semantics, Kratzer-style semantics proposes that the interpretation of deontic modals should incorporate an ordering source as a second conversational background (see \cite{Kratzer1991}).

When discussing instrumental obligation, as mentioned, it is essential to consider both the obligation itself and the goal it seeks to achieve. As noted by \cite{finlay2016oughts}, a key question arises: how does the goal impact the ordering source of obligation?

In this paper, we argue that the goal does not change the overall ordering but instead constrains a subset of the ordering when interpreting instrumental obligation.

We address this issue by applying the concept of priority structures used in preference reasoning \cite{liu2011reasoning,liu2011two}. As explained in \cite{van2014priority}, a priority structure offers an explanation for the ideal ordering, clarifying why possible worlds are ranked in a particular way. The priority structure reveals the reasons behind the ranking and generates the ordering accordingly.

The goal reflects the agent's current preferences, i.e., what they prefer and aim to achieve. Therefore, the goal should be formally represented by a preference ordering. However, the priority structure that generates this preference ordering may differ from the priority structure behind the ideal ordering, as one reflects \textit{personal} preferences and the other reflects \textit{objective} moral or normative principles. Thus, the interpretation of goals operates in an independent ordering system and does not influence the ideal ordering. In other words, an agent's goal does not affect the underlying reasons (i.e., the priority) for ranking ethical situations and, therefore, does not change the ideal ordering.

However, the goal will restrict the ordering used to interpret instrumental obligation. In a deontic model with an ideal ordering, obligations are interpreted in the most ideal part of the ordering. However, when interpreting instrumental obligation, we need consider the goal, the outcome of the action. As a result, formally, the obligation must be interpreted in worlds that are consistent with achieving the goal. This adds an additional restriction: the obligation is interpreted in the most ideal worlds that also are consistent with the agent's goal. In this way, the goal restricts the ordering source.

In this section, we discussed two semantic elements of instrumental obligation. First, causality explains the relationship between the obligatory action and the goal. Second, the ordering source interprets obligation in a deontic sense. These two elements provide the foundation for the semantics of instrumental obligation.

In the next section, we will introduce the formal mechanisms for each of these elements, which will serve as the basis for establishing the causal deontic models.

\section{Preliminaries: causal reasoning and priority structures}\label{prelimaries}

To formalize instrumental obligation and its reasoning based on causal information, we build on research in causal reasoning and deontic logic with priority structures. In this section, we outline the necessary preliminaries.

\subsection{Causal models}

We introduce structural equation models, as developed in \cite{Pearl95,halpern2000axiomatizing}, to represent causal inferences. Intuitively, a causal structure consists of two components: the causal variables and the causal relationships among them. Formally, these variables are defined by a signature $\mathcal{S}=(\mathcal{U},\mathcal{V},\Sigma)$, where $\mathcal{U}$ is a finite set of exogenous variables,\footnote{Causal models with infinite variables are explored in \cite{IbelingI19}.} $\mathcal{V}$ is a finite set of endogenous variables, and $\Sigma$ is the range of the variables.

In a structural equation model, causal influence among the variables is represented by a set of structural functions $\F$. For each endogenous variable $X$, $\F$ contains a function $f_X$, which determines the value of $X$ based on the other variables. Formally, a causal model is defined as a tuple $\langle \mathcal{S}, \mathcal{F}, \A \rangle$, where $\mathcal{S}=(\mathcal{U}, \mathcal{V}, \Sigma)$ is the signature, and $\mathcal{F}$ is a collection of functions $\{f_X\}_{X \in \mathcal{V}}$ with $f_X: ((\mathcal{U} \cup \mathcal{V}) \backslash \{X\}) \rightarrow \Sigma$. The function $f_X$ is called the structural equation for $X$, $\A$ is an assignment of values to the variables characterizing the actual state. Besides, let $W^\F$ denote the set of all assignments complied with structural functions.

In many studies of structural equation models, a causal model is often assumed to be \textit{acyclic} or \textit{recursive}, meaning the causal influence represented by $\F$ is acyclic\footnote{This means there is no sequence $X_1, ..., X_n$ such that for each $0 < k < n$, the value of $X_{k+1}$ depends on $X_k$ according to $\F$, and $X_1$ also depends on $X_n$.}. 

The basic language $\mathcal{L}$ of the causal model for causal reasoning is defined as follows:

\begin{Def}[Language $\mathcal{L}$] Let $\Sig=(\XV,\NV,\R)$,  formulas $\phi$ of the language $\mathcal{L}( \Sig)$ are given by

\begin{center}
     \[ \phi ::= X{=}x \mid \lnot \phi \mid \phi \land \phi \mid[\overrightarrow{V}=\overrightarrow{v}]\phi \]
\end{center}

where $X\in\XV\cup\NV$, $x\in\R$ and $\overrightarrow{V}=\overrightarrow{v}$ is a sequence of the form $V_1=v_1,...,V_n=v_n$ and  $\overrightarrow{V}\in\NV$. 
\end{Def}

  The atom $X{=}x$ means the value of the variable $X$ is $x$. For convenience, we will write both  $V_1=v_1,...,V_n=v_n$ and  $V_1=v_1\wedge...\wedge V_n=v_n$ as $\overrightarrow{V}=\overrightarrow{v}$. We call formulas of the form $X=x$ as atoms and let $\mathsf{At}$ denote the set of all atoms. 

 The formula $[\overrightarrow{V}=\overrightarrow{v}]\phi$ is an intervention formula that formalizes causal relationships through the concept of \textit{intervention}. An intervention is a hypothetical change in the actual state, along with the causal rules, that forces the value of certain endogenous variables to change. The result of an intervention is defined as follows.
\begin{Def}[Intervention]\label{Def: classical intervention}
Let $\langle \mathcal{S}, \mathcal{F}, \A \rangle$ be a causal model.
The causal model results from an intervention forcing the value of $\overrightarrow{X}$ to be $\overrightarrow{x}$ is defined as $\tuple{\S,\F_{\overrightarrow{X}=\overrightarrow{x}}, \A^\F_{\overrightarrow{X}=\overrightarrow{x}}}$ where:

\begin{itemize}
    \item the functions in $\F_{\overrightarrow{X}=\overrightarrow{x}} = \{ f'_{V} \mid V \in \NV \}$ are such that:  (i) for each $V$ \textbf{not in $\overrightarrow{X}$}, the function $f'_{V}$ is exactly as $f_{V}$, and (ii) for each $V {=} X_i \in \overrightarrow{X}$, the function $f'_{X_i}$ is a \emph{constant} function returning the value $x_i \in \overrightarrow{x}$ regardless of the values of all other variables.

    \item   $\A^\F_{\overrightarrow{X}=\overrightarrow{x}}$ is the unique\footnote{Since $\F$ is acyclic, $\F_{\overrightarrow{X}=\overrightarrow{x}}$ is also acyclic. Thus $\F_{\overrightarrow{X}=\overrightarrow{x}}$ has a unique solution with respect to each setting of exogenous variables.} assignment to $\F_{\overrightarrow{X}=\overrightarrow{x}}$ whose assignment to exogenous variables is identical with $\A$. Formally, $\A^\F_{\overrightarrow{X}=\overrightarrow{x}}(Y)
      $ is the unique assignment that satisfies the following equations:
    \[
     \A^\F_{\overrightarrow{X}=\overrightarrow{x}}(Y)
      =
      \left\{
        \begin{array}{ll}
          \A(Y)                                                            & \text{if}\; Y \in \XV \\
          f'_Y((\A^{\F}_{\overrightarrow{X}=\overrightarrow{x}})^{-Y}) & \text{if}\; Y \in \NV \\
        \end{array}
      \right.
    \]
   \end{itemize}     
 Note that $(\A)^{-X}$ denotes the sub-assignment of $\A$ to $(\XV\cup\NV) \backslash \{X\}$.
\end{Def}

Building on the concept of intervention, we can define the semantic clauses. The semantics for atomic and intervention formulas in the causal model $\M = \langle \Sig, \mathcal{F}, \mathcal{A} \rangle$ as outlined below. The Boolean connectives are defined in the standard manner.

\begin{itemize}
    \item  $ \tuple{\Sig, \F, \A}\models X=x$ iff $\A(X)=x$.
 
    \item $\tuple{\Sig, \F, \A}\models [\overrightarrow{X}=\overrightarrow{x}]\phi$ iff $\tuple{\Sig, \F_{\overrightarrow{X}=\overrightarrow{x}}, \A^\F_{\overrightarrow{X}=\overrightarrow{x}}}\models \phi$.
    
\end{itemize}

\subsection{Deontic models with priority}

In the preference logic developed by \cite{liu2008changing,liu2011two,liu2011reasoning}, priority structures formally represent the reasons underlying a preference ordering. In these structures propositions are ranked according to their importance, with the propositions themselves representing the reasons. From this priority structure, a preference ordering is generated, which reflects preferences among objects.

In deontic logic, the interpretation of obligation is based on an ideal ordering. Extending deontic models with priority structures, as seen in \cite{van2014priority}, provides a formal method to explain the generation of this ideal ordering. This process is often formalized through the concept of a \textit{P-graph} which defines a priority structure.

\begin{Def}[P-graph \cite{van2014priority}]  A $P$-graph is a tuple $G=\tuple{\Phi, \prec}$ such that

\begin{itemize}
    \item \( \Phi \) is a set of propositions or normative statements.
    \item \( \prec \) is a strict ordering on \( \Phi \), indicating that \( \psi \prec \phi \) means that proposition \( \psi \) is of lower priority than \( \phi \).
\end{itemize}
    
\end{Def}

Each proposition in a P-graph implies the propositions that follow it in the ordering. For example, if \( \phi_1 \prec \phi_2 \), then \( \phi_2 \) holds whenever \( \phi_1 \) holds. This reflects the idea that the highest-priority must be satisfied first, while lower-priority ones are secondary.


In addition, a betterness relation based on the priority structure can be defined:

\begin{Def}[Betterness from P-graphs \cite{van2014priority}]~ Let $G=\tuple{\Phi, \prec}$ be a P-graph, \( S \) be a non-empty set of possible worlds, the betterness relation $\leq_G\subseteq S^2$ is defined as:

\begin{center}
  
$s\preceq_G s' $ iff for any $\phi\in \Phi$, $s \models \phi \Rightarrow s' \models \phi$

\end{center}

\end{Def}

This means that a world \( s \) is at least as good as state \( s' \) if, for every normative proposition \( \phi \) in the priority structure, whenever \( s \) satisfies \( \phi \), so does \( s' \). The set of maximal elements in this ordering represents the ideal worlds under the given priority structure. By comparing the satisfaction of propositions in \( \Phi \) across different possible worlds, we can determine which worlds are more ideal. 

\section{Causal deontic models}\label{CDM}

In the previous sections, we explored causality and ordering as the two significant semantic foundations of instrumental obligation. Causality captures the instrumental meaning, while ordering provides the basis for interpreting obligation. In this section, we establish a model to formalize instrumental obligation by integrating these two foundations, combining causal models with deontic model having priority structures.

The basic idea is to extend the causal model by a priority structure, as developed by \cite{xie2024logic} for analyzing desire. In our approach, we employ this method by adding a P-graph into the causal model, from which an ideal ordering is derived. Unlike \cite{xie2024logic}, where the P-graph is defined over variables, our model defines the P-graph among causal atoms to better capture the priority of normative propositions.

\begin{Def}[Priority ordering among atoms] 
Let $\mathsf{At}$ be the set of atomic formulas. A priority ordering $P$ over a subset of variables $\Phi \subseteq \mathsf{At}$ is a sequence $\langle \Phi, \ll \rangle$. The notation $Y=y \ll X=x$ indicates that $X=x$ is strictly more important than $Y=y$. 
\end{Def}

From this priority ordering, we can induce an ideal ordering among assignments.

\begin{Def}[Ideal ordering from $P$]\label{def:ideal ordering fromP}
Given a priority ordering $P$ and a causal model $\M$, the ideal ordering $\leq_P \subseteq W^\F \times W^\F$, induced by $P$ and $\M_D$, is defined as follows: for any $\A,\A' \in W^\F$, $\A \leq_P \A'$ whenever: 
\begin{center}
For any $X=x \in \Psi$: $\A(X)=x \Rightarrow \A'(X)=x$, or there exists $Y=y \in \Psi$ such that $\A(Y)\neq y$, $\A'(Y)=y$, and for all $Z=z \in \Psi$, $\A(Z)= z$ and $\A'(Z) \neq z$ implies $Z=z \ll Y=y$.
\end{center}
\end{Def}

An assignment $\A$ is considered a best world if there is no $\A'$ such that $\A \leq_P \A'$ and $\A \neq \A'$. Additionally, $\A < \A'$ if $\A \leq_P \A'$ and $\A' \not\leq_P \A$. In fact, each assignment $\A$ can be seen as a possible distribution of values for the variables (except for the fact that we require them to comply with the structural equations). In this sense, each assignment can be regard as a possible world: they are all descriptions of possible states of the world (the values).

A causal deontic model can then be defined by including the priority ordering $P$ into the causal model.

\begin{Def}[Causal deontic models]
A causal deontic model is a tuple $\M_D = \langle \S, \F, P, \A \rangle$, where $\S$, $\F$, and $P$ are defined as above. The \textit{ideal ordering} $\leq_P$ is induced by $P$, and $\A \in W^\F$, where $W^\F$ is the set of assignments that comply with $\F$.
\end{Def}

The result of an intervention for a causal deontic model can be defined in the original way, as Def \ref{Def: classical intervention}. In other words, the extension with priority ordering does not affect the intervention process. Formally, the ideal ordering $\leq^{\X=\x}_P$ in the model after an intervention by $\X=\x$ is defined as follows: $\A_1^{\X=\x} \leq^{\X=\x}_P \A_2^{\X=\x}$ whenever $\A_1 \leq_P \A_2$.

Let us now consider an example to illustrate how this model functions. 

\begin{quote}
    \textit{John is overweight. The doctor advise him to lose weight. Then there are two options for him: one is doing more sports, the other one is taking some weight-loss pills. But the pills would bring some side-effects. So, in order to lose weight, John ought to do more sports.}
\end{quote}

Let the letters $A$, $B$, $C$, and $D$ represent the events ``doing sports", ``taking pills", ``losing weight" and ``experiencing side effects," respectively. When an event occurs, it is assigned the value 1, otherwise it is assigned the value 0. For example, ``$A=1$" indicates that doing sports is taking place. We can define the following structural functions to reflect the causal relationships between these events.

\noindent
\begin{minipage}[b]{0.6\textwidth}
    \centering
    \begin{tabular}{|l l | l |}
    \hline
    $f_A$: & $f_A(\A^{-A})=1$ iff $\A(U_A)=1$;   \\
    $f_B$: & $f_B(\A^{-B})=1$ iff $\A(U_B)=1$;    \\
    $f_C$: & $f_C(\A^{-C})=1$ iff $\A(A)=1$,    \\ 
     &  or  $\A(B)=1$;  \\ 
    $f_D$: & $f_D(\A^{-D})=1$ iff $\A(B)=1$ \\
    \hline
    \end{tabular}
    \vspace{1.4em}
    \par\vspace{5pt}  
    \textbf{Structure functions}  
\end{minipage}%
\hspace{0.05\textwidth}  
\begin{minipage}[b]{0.35\textwidth}
    \centering
    \begin{tikzpicture}[->,>=stealth',shorten >=1pt,auto,node distance=2cm,
                    semithick]
        \tikzstyle{state}=[circle,draw]
        \node    (A)    {$A=1$}; 
        \node     (B)   [right=4em of A]    {$B=0$}; 
        \node   (UA)   [above=2em of A]   {$U_A=1$};
        \node   (UB)   [above=2em of B]   {$U_B=0$ };
        \node  (C)   [below left=2em of B]   {$C=1$};
        \node  (D)   [below =2em of B]   {$D=0$};
        
        \path
        (B) edge              node {} (C)
        (UB) edge              node {} (B) 
        (UA) edge              node {} (A)
        (B) edge              node {} (D)
        (A) edge              node {} (C);
    \end{tikzpicture}
    \vspace{-0.5em}
    \par\vspace{5pt}
    \textbf{Causal relationships}  
\end{minipage}

\vspace{0.7em}

From the table and figure, we observe that the values of both $A$ and $B$ are determined by exogenous variables. Additionally, both $A=1$ and $B=1$ can lead to $C=1$, while $B=1$ will bring out $D=1$.

In this example, we assume that experiencing side effects is the worst outcome, followed by not losing weight. Losing weight is a better outcome, and the best scenario is losing weight without any side effects. Thus, we can have the following priority ordering:

\begin{center}
  $D=1\ll C=0 \ll C=1 \ll D=0$   
\end{center}

Based on the definition of the induced ideal ordering, we can rank all possible assignments to the variables. We present a fragment of this ordering as follows:

\begin{table}[]
    \centering
    \renewcommand{\arraystretch}{1} 
    \begin{tabular}{|>{\centering\arraybackslash}p{1cm}|>{\centering\arraybackslash}p{2cm}|>{\centering\arraybackslash}p{2cm}|>{\centering\arraybackslash}p{2cm}|>{\centering\arraybackslash}p{2cm}|}
    \hline
    & $A$ & $B$ & $C$ & $D$ \\
    \hline
    $\A_1$ & 0  & 0 & 1 & 0 \\
    \hline
    $\A_2$ & 1 & 0 & 1 & 0 \\
    \hline
    $\A_3$ & 0 & 0 & 0 & 0 \\
    \hline
    $\A_4$ & 1 & 1 & 1 & 1 \\
    \hline
    $\A_5$ & 0 & 1 & 1 & 1 \\
    \hline
    \end{tabular}
    \\
    \vspace{0.5em}
    $Min_\leq^P= \{A_4,\A_5\}$, $Max_\leq^P= \{A_1,\A_2\}$
    \vspace{0.3em}
    \caption{Ideal ordering of the assignments}
    \label{ioass}
\end{table}

Table \ref{ioass} shows that the actual state is $\A_3$, where nothing occurs before behaving. Both $\A_1$ and $\A_2$ represent better states than the current one, as they involve weight loss without side effects. In contrast, $\A_4$ and $\A_5$ are worse than the current state due to the presence of side effects. Although there are more worse assignments where only side effects occur, they are not shown in the table.

The instrumental obligation can be expressed by the ordering in the model. Intuitively the idea is that the goal, such as ``no side effects" and ``weight loss," is represented in the priority ordering. We then evaluate the action by making an intervention on the actual assignment and checking whether the intervention leads to an assignment where the goal is achieved, and if the result is the best possible assignment after the intervention. If so, it will indicate that the action ought to be taken.

This approach captures the instrumental meaning in our model by defining obligation through intervention. In the following section, we will formally define the modality of instrumental obligation.

\section{Instrumental obligation as intervention} \label{IOI}

In this section, we define the modality of instrumental obligation. As discussed in Section \ref{SFIO}, the instrumental meaning is grounded in causal relationships. In causal models, intervention formulas directly represent these causal connections. Therefore, we aim to reduce the instrumental obligation operator to intervention formulas.

The basic language $\mathcal{L}_D$ of a causal deontic model $\mathcal{M}_D$ extends the basic causal reasoning language.

\begin{Def}[Language $\mathcal{L}_D$]
Let $\Sigma = (\mathcal{U}, \mathcal{V}, \mathcal{R})$. The formulas $\phi$ of the language $\mathcal{L}_D$ are defined as follows:
\[
\phi ::= X=x \mid \lnot \phi \mid \phi \land \phi \mid [\overrightarrow{V}=\overrightarrow{v}]\phi \mid X=x \prec Y=y \mid \phi^{\u}
\]
where $X \in \mathcal{U} \cup \mathcal{V}$, $x \in \mathcal{R}$, and $\overrightarrow{V} = \overrightarrow{v}$ is a sequence of assignments of the form $V_1 = v_1, \dots, V_n = v_n$ with $\overrightarrow{V} \in \mathcal{V}$ and $\u \subseteq \mathcal{R}(\mathcal{U})$.
\end{Def}

We extend the basic causal language by introducing the ``priority" formula $X=x \prec Y=y$ and the formula $\phi^{\u}$. The ``priority" formula means that $Y=y$ has higher priority than $X=x$. The formula $\phi^{\u}$ is also novel which states that $\phi$ is true under the assignment where $\mathcal{U} = \u$. Since $\mathcal{U}$ represents exogenous variables, each $\u$ determines a unique assignment when the structural functions are recursive.

\begin{Def}[Truth conditions for $\mathcal{L}_D$]
Let $\mathcal{M}_D = \langle \Sig, \mathcal{F}, P, \A \rangle$ be a causal deontic model. The truth conditions for the new operators are as follows:
\begin{itemize}
    \item $\mathcal{M}_D \models X=x \prec Y=y$ iff $X=x \ll Y=y$
    \item $\mathcal{M}_D\models \phi^{\u}$ iff $\A' \models \phi$ where $\A'(\mathcal{U}) = \u$
\end{itemize}
\end{Def}

Both operators are global. It is worth noting that, ``$\u$" acts as a nominal in the model, similar to \textit{Hybrid Logic} (see \cite{BlackburnSeligman1995}). Since the causal models we discuss are recursive, the values of all exogenous variables $\mathcal{U}$ (denoted by $\U$) fully determine the assignment. Each assignment can thus be uniquely labeled by a specific $\u$. Furthermore, we can express the ordering between assignments as:
\[
\U = \u \leq \U = \u'
\]
This formula indicates that the assignment with $\U = \u'$ is at least as good as the assignment with $\U = \u$. Based on our definition of the induced ideal ordering, this can be formally captured in the language as follows:
\begin{align}
\U = \u &\leq \U = \u' :=  \bigwedge\limits_{X\in \mathcal{U}\cup \mathcal{V}, x\in \mathcal{R}(X)} 
(X = x^{\u} \wedge X = x^{\u'}) \nonumber \\
&\vee \bigvee\limits_{Y\in \mathcal{U}\cup \mathcal{V}, y \in \mathcal{R}(Y)} 
\left( (Y = y^{\u'} \wedge \neg Y = y^{\u}) \wedge \right. \nonumber \\
&\left. \bigwedge\limits_{Z\in \mathcal{U}\cup \mathcal{V}, z\in \mathcal{R}(X)} 
(Z = z^{\u} \wedge \neg Z = z^{\u'} \rightarrow Z = z \prec Y = y) \right)
\end{align}
The formula (1) characterizes the conditions for $\U = \u' \leq \U = \u $ to hold: let $\A,\A'$ denote the assignments with $\A(\U)=\u$ and $\A'(\U)=\u'$ respectively, then $\A' \leq \A$ iff (i) $\A(X)=\A'(X)$ for all $X\in \XV\cup \NV$; or (ii) if there is $Y\in \XV\cup \NV$ such $\A'(Y)=y\neq \A'(Y)$, then for any $Z\in \XV\cup \NV$ with $\A(Z)=z\neq \A'(Z)$, $Z=z$ is strictly less important than $Y=y$. The expression matches our definition of ideal ordering from $P$ in Def. \ref{def:ideal ordering fromP}. 

We then extend this to operates interventions:

\begin{align}
[\Y=\y]\U=\u &\leq [\Z=\zi]\U=\u' := \bigwedge\limits_{X\in \mathcal{U}\cup \mathcal{V}, x\in \mathcal{R}(X)} 
\left( ([\Y=\y]X=x)^{\u} \wedge ([\Z=\zi]X=x)^{\u'} \right) \nonumber \\
&\vee \bigvee\limits_{W\in \mathcal{U}\cup \mathcal{V}, w \in \mathcal{R}(W)} 
\left( \left( ([\Z=\zi]W=w)^{\u'} \wedge \neg ([\Y=\y]W=w)^{\u} \right) \wedge \right. \nonumber \\
&\left. \bigwedge\limits_{E\in \mathcal{U}\cup \mathcal{V}, e \in \mathcal{R}(E)} 
\left( ([\Y=\y]E=e)^{\u} \wedge (\neg [\Z=\zi]E=e)^{\u'} \rightarrow E=e \prec W=w \right) \right)
\end{align}

This expresses the relative ordering of assignments after an intervention, including the changes to the relevant variables. Formula (2) follows the idea of Formula (1), except that we use the intervention operator to indicate the worlds after intervening.

Based on these definitions, the instrumental obligation modality, denoted as $\IO$, can be defined as a derived operator in the language $\mathcal{L}_D$:
\begin{align}
\IO(X=x : \Y=\y)^{\u} := &\ (\neg X=x)^{\u} \wedge ([\Y=\y]X=x)^{\u} \nonumber \\
&\wedge \bigwedge\limits_{\Z \subseteq \mathcal{V}, \zi \subseteq \mathcal{R}(\Z)} 
\left( ([\Z=\zi]X=x )^{\u} \rightarrow [\Z=\zi] \U=\u \leq [\Y=\y] \U=\u \right)
\end{align}
This formula is read as: ``In order to $X=x$, it ought to be that $\Y=\y$." Semantically, in the actual world $\A$ where $\A(\U)=\u$, the expression $(\neg X=x)^{\u}$ means that $X$ does not occur in $\A$, indicating that $X=x$ can be a goal. The idea is straightforward: the goal is something that has not yet been realized in the real world. The formula $([\Y=\y]X=x)^{\u}$ means that executing the action $\Y=\y$ will lead to $X=x$ in $\A$, capturing the causal relationships between the action and the goal. The final conjunction asserts that, in the current world, no alternative action will result in a better outcome than executing $\Y=\y$.

Revisiting the example, we can derive that in the model, the formula $\IO(C=1 : A=1)$ is valid. Intuitively, intervening on $A=1$ leads to the state $\A_2$. Although this is not the best possible world in the entire ordering, it is the best world given the current state by doing $A=1$, and $A=1$ leads to $C=1$. Therefore, $A=1$ is instrumentally obligatory. However, $B=1$ is not obligatory, as intervening on $B=1$ leads to a state where $D=1$, which is worse than the actual world. Thus, $B=1$ is not an obligation. 

Building on this formalization, we can define permission under instrumental meaning accordingly, denoted as $\IP$, which states that ``in order to $X=x$, you can execute $Y=y$".

\begin{align}
\IP(X=x:\Y=\y)^{\u} := &\ (\neg X=x)^{\u} \wedge ([\Y=\y]X=x)^{\u} \nonumber \\
&\wedge \bigwedge\limits_{\y' \subseteq \mathcal{R}(\Y)} \left( ([\Z=\zi]X=x)^{\u} \rightarrow [\Y=\y'] \U=\u' \leq [\Y=\y] \U=\u' \right)
\end{align}

The distinction between $\IO$ and $\IP$ lies in the final conjunction. For instrumental obligation, it indicates that executing the action is better than any alternative. In contrast, for permission, it suggests that while executing $Z=z$ might also lead to $X=x$ and result in a better world, executing $Y=y$ is still ideal than to not doing so when trying to achieve $X=x$. Intuitively, this means that $Y=y$ is one possible method to reach $X=x$, but not the only one\footnote{This notion of permission differs from standard deontic logic, where permission is typically defined as the dual of obligation. Instead, it bears some resemblance to the concept in game models of deontic logic (\cite{RoyAnglbergerGratzl2012}), where obligation is considered the weakest form of permission.}.

\section{A logic for instrumental obligation}\label{CIO}

In this section, we give a Hilbert-style calculus for reasoning about the instrumental obligation, $\mathsf{CIO}$, and prove its soundness, completeness, and address its computational complexity. 

\subsection{Soundness and completeness}

The calculus $\mathsf{CIO}$ includes the following rules and axioms:

\begin{itemize}
    \item[(P)]Propositional  tautology. 
    \item[MP] From $\varphi_1$ and $\varphi_1 \rightarrow \varphi_2$ infer $\varphi_2$  
    \item[Nec] From $\varphi$ infer $[\overrightarrow{X}=\overrightarrow{x}]\varphi$ and $\varphi^{\u}$. 
    \item[$A_1$]  $ [\overrightarrow{X}{=}\overrightarrow{x}]Y{=}y \rightarrow \lnot [\overrightarrow{X}{=}\overrightarrow{x}]Y{=}y'$  for $y,y' \in \{0,1\}$ with $y \neq y'$ 
    \item[$A_2$] $ \bigvee_{y \in R(Y)} [\overrightarrow{X}{=}\overrightarrow{x}]Y{=}y$
    \item[$A_3$] $ ( [\overrightarrow{X}{=}\overrightarrow{x}]Y{=}y \land [\overrightarrow{X}{=}\overrightarrow{x}]Z{=}z ) \rightarrow [\overrightarrow{X}{=}\overrightarrow{x},\overrightarrow{Y}=\overrightarrow{y}]Z{=}z$  
    \item[$A_4$] $[\overrightarrow{X}{=}\overrightarrow{x},Y{=}y]Y{=}y$
    \item[$A_5$]  $ (X_0 \syndirparof X_1\wedge \cdots \wedge X_{k-1} \syndirparof X_k) \rightarrow \lnot (X_k \syndirparof X_0)$, $X_0,...,X_k$ are distinct variables in $\NV$\footnote{$Y \rightsquigarrow Z $ is an abbreviation of $ \bigvee_{\overrightarrow{X} \subseteq  \mathcal{V}\backslash\{Y,Z\}, y,y',z,z'\in\{0,1 \},y\neq y',z\neq z', Y\neq Z}([\overrightarrow{X}=\overrightarrow{x},Y=y]Z=z' \wedge [\overrightarrow{X}=\overrightarrow{x}, Y=y']Z=z) $ which indicates $Y$ has causal influence on $Z$.}
    \item[$A_\lnot$] $[\overrightarrow{X}{=}\overrightarrow{x}]\neg\varphi \leftrightarrow \lnot [\overrightarrow{X}{=}\overrightarrow{x}]\varphi$ 
    \item[$A_\land$ ] $[\overrightarrow{X}=\overrightarrow{x}]\varphi_1\wedge\varphi_2 \leftrightarrow [\overrightarrow{X}=\overrightarrow{x}]\varphi_1 \land [\overrightarrow{X}=\overrightarrow{x}]\varphi_2)$ 
    \item[$A_{[][]}$] $[\overrightarrow{X}=\overrightarrow{x}][\overrightarrow{Y}=\overrightarrow{y}]\varphi \leftrightarrow [\overrightarrow{X'}=\overrightarrow{x'},\overrightarrow{Y}=\overrightarrow{y}]\phi$, with $\overrightarrow{X'}=\overrightarrow{x'}$ the subassignment of $\overrightarrow{X}=\overrightarrow{x}$ for $\overrightarrow{X'}:=\overrightarrow{X}\backslash \overrightarrow{Y}$
    \item[Asym] $X=x \prec Y=y \rightarrow \neg (Y=y \prec X=x)$ 
    \item[Trans]  $X=x \prec Y=y \wedge Y=y \prec Z=z \rightarrow X=x \prec Z=z $ 
    \item[$G_\wedge$] $(\varphi \wedge \psi)^{\u} \leftrightarrow (\varphi^{\u} \wedge \psi^{\u})$ 
    \item[$G_\neg$] $\neg(\varphi)^{\u} \leftrightarrow (\neg\varphi)^{\u}$
    \item[$G_{\u}$] $(\varphi^{\u})^{\u'}\leftrightarrow \varphi^{\u}$
    \item[$G_{[]}$] $([\overrightarrow{X}=\overrightarrow{x}]\phi)^{\u})\leftrightarrow [\overrightarrow{X}=\overrightarrow{x}](\phi^{\u})$
    \item[Self]$(\U=\u)^{\u}$
    \item[Incl] $\U={\u} \wedge \varphi \rightarrow \phi^{\u}$
\end{itemize}

The first kind of the axioms, which includes $A_1$ to $A_5$, $A_\lnot$, $A_\land$ and $A_{[][]}$, is directly from the system $AX_{rec}$ in \cite{halpern2000axiomatizing}. They describe how the intervention operator works in the causal structure. $A_1$ to $A_4$ express the functionality of intervention. $A_5$ guarantees the causal influence is acyclic. $A_\lnot$, $A_\land$ and $A_{[][]}$ are reduction axioms for Boolean operators.  \\
The second kind of axioms includes (Asym) and (Trans). The two axioms characterize the priority ordering is a strict partial order. \\
For the third kind of axioms: ($G_\wedge$), ($G_\neg$), ($G_{\u}$), ($G_{[]}$) are reduction axioms for $\u$. (Self) and (Incl) express the essential property of $\u$: $\u$ is the value of all exogenous variables.

\begin{theorem}
    $\mathsf{CIO}$ is sound and complete w.r.t causal deontic models. 
\end{theorem}

The soundness is easy to be verified. For completeness, it is sufficient to show that for any maximal $\mathsf{CIO}$-consistent set, there is a causal-denotic model modelling it. 

\begin{definition}[Canonical model]
    Given a maximal $\mathsf{CIO}$-consistent set $\Gamma$, the canonical model $\M^\Gamma= (\Sig, \F^\Gamma, P^\Gamma)$ is defined as follow: 
    \begin{itemize}
        \item For each variable $V \in \NV$, the structural function $\F^\Gamma= \{f^\Gamma_V \}$ is defined in such a way: 
        $$
        f^\Gamma_V (\U=\u,\X=\x)=v ~\text{iff}~ ([\X=\x]V=v)^{\u}\in \Gamma
        $$ where $\U$ are all exogenous variables and $\X$ are all endogenous variables in $\NV \backslash \{ V \}$.  
    \item $P^\Gamma=(\Phi^\Gamma, <^\Gamma)$ where $\Phi^\Gamma=\mathsf{At}$ and $<^\Gamma$ is defined as: 
    $$
         X=x <^\Gamma Y=y ~\text{iff}~ X=x\prec Y=y \in \Gamma
    $$
    \item  $\A^{\U=\u}$ is defined as $\A^{\U=\u }(X)=x$ iff $(X=x)^{\u}\in \Gamma$ and $W^\Gamma = \{\A^{\U=\u }: \u \in R(\U) \}$.
    \end{itemize}
\end{definition}
Now we need to verify that the canonical model that we construct is well-defined: 
$A_1, A_2,A_5$ and (Incl) guarantees that there is a unique $v$ such that $([\X=\x]V=v)^{\u} \in \Gamma$. (Asym) and (Trans) guarantee that $<^\Gamma$ is a partial order. $A_1,A_2$ and (Gob) guarantee that each $\A^{\U=\u}$ is well-defined. Since $\Gamma$ is maximal consistent set, there is an atom $\U=\u\in \Gamma$. Then by (Incl), there is an $\A^{\U=\u}\in W^{\Gamma}  $ such that $\A^{\U=\u }(X)=x$ iff $(X=x)\in \Gamma$. Let $\A^\Gamma$ denote it.

\begin{lemma}[Truth Lemma]\label{lem:truth}
    For any $\phi \in \mathcal{L}(\Sig)$, $\phi \in \Gamma$ iff $\M^\Gamma, \A^\Gamma \models \phi$. 
\end{lemma}
\begin{proof}
    Induction on $\phi$. The based case is trivial. The case when $\phi$ is $X=x \prec Y=y$ also directly from the definition of $<^\Gamma$. If $\phi$ is $[\X=\x]\phi$. By $A_{\neg}$, $A_{\wedge},A_{[][]}$, it is sufficient to show that $[\X=\x]Y=y \in \Gamma$ iff $\M^\Gamma, \A^\Gamma \models [\X=\x]Y=y$. Suppose that $[\X=\x]Y=y \in \Gamma$. By (Incl), we have  $([\X=\x]Y=y)^{\u}   \in \Gamma$. Then by the definition of $f_Y^\Gamma$, we have $f^\Gamma_Y(\U=\u,\X=\x)=y$ which implies $\M^\Gamma, \A^\Gamma \models [\X=\x]Y=y$. One the other hand, if $[\X=\x]Y=y \notin \Gamma$, then by $A_1$ and $A_2$, there is $y'\neq y$ such that $[\X=\x]Y=y' \in \Gamma$. Then by the same argument, we have $f^\Gamma_Y(\U=\u,\X=\x)=y'$ which implies $\M^\Gamma, \A^\Gamma \not\models [\X=\x]Y=y$.  \\
    When $\phi$ is $\psi^{\u'}$: the proof goes by induction on $\psi$. The based case is straightforward from the definition of $\A^{\U=\u'}$. The rest cases can be obtained directly by $G_{\neg}$, $G_{\wedge},G_{\u}$ and $G_{[]}$. 
\end{proof}

By Lemma \ref{lem:truth}, the completeness is obtained:
\begin{theorem}[Completeness]
      $\mathsf{CIO}$ is complete w.r.t causal deontic models. 
\end{theorem}

\subsection{Complexity}
In this section, we explore the complexity of the logic, let $|\phi|$ denote the length of $\phi$ which is viewed as a string of symbols:
\begin{lemma}\label{lem:reduce}
    Given a signature $\Sig{=}(\XV,\NV,\R)$.  Let $\langle \phi \rangle{=} \{X \in \XV \cup \NV :~ X~\text{appears in}~ \phi   \}$. Let $\Sig_\phi=(\{U^*\}, \NV_\phi, \R_\phi)$, where $\NV_\phi= \langle \phi \rangle $; $U^*$ is a fresh variable which does not occur in $\XV$ or $\NV$, $\R_\phi(X)=\R(X)$ for all $X\in \NV_\phi$ and $\R_\phi(U^*)$ consists of all tuples in $\times_{U\in \XV}\R(U)$, that is, the value of $U^*$ represents the distribution of values of $\XV$.

    A $\mathcal{L}(\Sig)$ formula $\phi$ is satisfiable in a model based on a signature $\Sig$ iff it is satisfiable in a model based on $\Sig_\phi$. 
\end{lemma}
\begin{proof}

From left to right:

We construct a finite model $M {=} \tuple{\Sig_\phi, \F_\phi,P_\phi} $ based on $\Sig_\phi$. First, we define $\F_\phi$: since $\phi$ is satisfiable in $M$, there exists an ordering $\vartriangleleft$ among the endogenous variables in $\NV$ such that if $X \vartriangleleft Y$, then the value of $F_X$ is independent of the value of $Y$. Let $Pre(X){=}\{ Y \in \XV ~|~ Y \vartriangleleft X \}$. For convenience, we take the values of variables in $\XV \cup Pre(X)$ as parameters of $f_X$ (Since $X$ is independent of the variables in $ \NV\backslash Pre(X)$, the values of those variables has no effect on $f_X$).

Then for any endogenous variable $X \in \XV $, we define $f'_X$ as follows: Induction on $\vartriangleleft$. Let $f'_X(\overrightarrow{u},\overrightarrow{a}){=}f_X(\overrightarrow{u},\overrightarrow{b})$ where $\overrightarrow{a}$ is the values of variables  $\NV_\phi \backslash \{X\}$ and $\overrightarrow{b}$ is the values of variables in $Pre(X)$, $\overrightarrow{u}$ is the value of variables in $\XV$. If $X$ is $\vartriangleleft$-minimal, then $f'_X(\overrightarrow{u},\overrightarrow{a}){=}f_X(\overrightarrow{u})$.
Inductive: let $f'_X(\overrightarrow{u},\overrightarrow{a}){=}f_X(\overrightarrow{u},\overrightarrow{b})$ where $\overrightarrow{a}$ is the values of variables  $\NV_\phi \backslash \{X\}$ and $\overrightarrow{b}$ is the values of variables in $Pre(X)$. For any $Y \in Pre(X)$, if $Y \in Pre(X) \cap \NV_\phi $, then the value of $Y$ in $\overrightarrow{b}$ is the value of $Y$ in $\overrightarrow{a}$. If $Y \in Pre(X) \backslash \NV_\phi $
, then the value of $Y$ in $\overrightarrow{b}$ is $f'_Y(\overrightarrow{u},\overrightarrow{a})$(By I.H, $f'_Y(\overrightarrow{u},\overrightarrow{a})$ has been defined).

Let $\F_\phi{=}\{f^\phi_X : X \in \langle \phi 
\rangle \}$, we define $P_\phi=\tuple{\Phi_\phi, \ll_\phi}$ as follow: $Y=y \ll_\phi X=x$ iff $Y=y\prec X=x $ appears in $\phi$, $\Phi_\phi$ is the collection of the subformulas of the formulas which is form of  $Y=y \ll_\phi X=x$ appearing in $\phi$. Now we the causal model $M$ has been defined.

Then we induction on $\phi$. The cases without $\prec$ and $\phi^{\u}$ are similar in \cite{halpern2000axiomatizing}. If $\phi$ is $Y=y\prec X=x$, then it follows from the definition of $\ll_\phi$ directly. If $\phi$ is $\psi^{\u}$. Since $\psi^{\u}$ is satisfiable based on $\Sig$, then there is an assignment $\A$ in $M$ such that $\A(U^*)=\u$ and $M,\A \models \psi$ by I.H. Hence, $\psi^{\u}$ is satisfiable in $M$. 

From right to left: 

Given a model $M=\tuple{\Sig_\phi, \F, P}$, we can define a model $M'=\tuple{\Sig, \F' , P'}$ as follows: let $f'_X$ be a constant for $X\in \XV \backslash \XV_\phi$; if $X \in \XV_\phi$, define $f'_X(\u,\zi,\y)=f_X(\u,\zi)$, where $\u \in \R(U^*)$, $\zi$ is the values of the variables in $\XV_\phi \backslash \{ X \}$ and $\y$ is the values of $\XV \backslash \XV_{\phi}$; if $\u \notin \R(U^*)$, let $f'_X(\u,\zi,\y)$ be an arbitrary constant. Let $P=P_\phi$. Then it is easy to see that if  $\phi$ is satisfiable in $M$, it is also satisfiable in $M'$ by induction on $\phi$.  

\end{proof}

\begin{theorem}
    Given a formula $\phi$ and a signature $\Sig$, the problem of deciding if $\phi$ is satisfiable with respect to the class of causal deontic models is NP-complete.
\end{theorem}
\begin{proof}
The NP-lower bound can be verified by encoding the satisfiability (SAT) problem for propositional logic into the SAT problem for $\mathcal{L}(\Sig)$. \cite{halpern2000axiomatizing} proved this results w.r.t to basic causal logic. Since our logic is an extension of the basic causal logic, the proof in \cite{halpern2000axiomatizing} also works for our logic. 

For the NP upper bound, we verify whether $\phi$ is satisfiable by guessing a causal deontic model. What we need to construct are the set of structure functions $\F$ and the priority ordering $P$. By Lemma \ref{lem:reduce}, it suffices to check whether $\phi$ is satisfiable in a $\Sig_\phi$-model. Then the numbers of the variables that we need to consider are limited to $|\phi|$. Similar to Lemma \ref{lem:reduce}, $P$ can be described directly from the information of $\phi$.

Let $R$ be the set of all $\X=\x$ where $[\X=\x]$ appears in $\phi$.
We say that two models $M$ and $M'$ agree on $R$ if for each $\X=\x \in R$, given context $\u$, $M_{\X=\x}$ and $M'_{\X=\x}$ have the same unique solution. It follows that if $M$ and $M'$ agree on $R$, then both $M_{\X=\x}$ and $M'_{\X=\x}$ satisfy $\phi$ in context $\u$ or neither do.
Hence, all we need to know about models is how it deals with the assignments in $R$. 

For each $\X=\x \in R$, we guess a vector $\overrightarrow{v}(\X=\x)$ of values for the endogenous variables. Given this guess, it is easy to check whether $\phi$ is satisfied in a model where these guesses are the solutions to the equations. 

The last thing we need to do is showing that there exists a causal deontic model in $\Sig_\phi$ such that the relevant equations have these solutions. We first guess an ordering $\vartriangleleft$ on the variables ( Similar to Lemma \ref{lem:reduce}). Then verify whether the guessed $\overrightarrow{v}(\X=\x)$ are compatible with $\vartriangleleft$. If the solutions are compatible with $\vartriangleleft$, then we can define the structure functions $f_X$ for $X \in \XV$ such that all the equations hold and $f_X$ is independent of the value of $Y$ if $X \vartriangleleft Y$. Hence, the SAT problem for $\mathcal{L} (\Sig)$ is in NP. As a consequence, it is NP-complete. 
\end{proof}

\section{Conclusion}\label{conclu}

In this paper, we developed a logic to formalize the concept of instrumental obligation. We began by constructing a causal deontic model that extends a causal model with a priority structure. Causal relationships are used to capture the instrumental meaning of obligation, while the priority structure derives an ideal ordering, which grounds the deontic meanong of obligation. 

In this model, instrumental obligation is defined as a derived notion, which is represented through intervention. It means that an action is obligatory if, after performing the action, the goal is achieved in a best way. We also defined permission in this framework that an action is permitted if it is a good to achieve the goal. 

Finally, we introduced a logical system for the model, proving its soundness and completeness, and showing that the logic is NP-complete.

In future work, we plan to formalize the notion of multiple goals in this model and consider how to model desire in the logic.

%
%
%
 \bibliographystyle{splncs04}
 \bibliography{awpl}

\end{document}